\documentclass[a4 paper,12pt]{article}
\usepackage[english]{babel}
\usepackage[T1]{fontenc}
\usepackage{amsfonts}
\usepackage{latexsym}
\usepackage{amsmath}
\usepackage{amsthm}
\usepackage{amssymb}

\title{On modulo-recurrence and window complexity in infinite words}
\author{Julien Cassaigne \and Idrissa Kabor\'e 
}

\usepackage{fancyhdr}
\newtheorem{theorem}{Theorem}[section]
\newtheorem{lemma}{Lemma}[section]
\newtheorem{remark}{Remark}[section]
\newtheorem{definition}{Definition}[section]
\newtheorem{proposition}{Proposition}[section]
\newtheorem{corollary}{Corollary}[section]

\newcommand{\ie}{\emph{i.e.}}

\begin{document}
\renewcommand{\refname}{References}

\maketitle
\setcounter{footnote}{0}
\thanks

\begin{center}
$^1$ Institut de Math\'ematiques de Marseille, 163 Av. Luminy,\\ Case 907
13288 Marseille, France\\
julien.cassaigne@math.cnrs.fr\\
$^2$ UFR-Sciences Exactes et appliquées, Universit\'e Nazi  Boni,\\ 01 BP 1091 Bobo-Dioulasso 01, Burkina Faso\\ ikaborei@yahoo.fr
\end{center}

%\begin{center}
\noindent\textbf{Abstract.} In this paper, we introduce the notions of  uniform modulo-recurrence and strong modulo-recurrence. We show that Sturmian words and maximal complexity words are strongly  modulo-recurrent. Next, a relationship between the window complexity and the classical complexity of the Thue-Morse word is established. We provide an  aperiodic recurrent word such that the window complexity is bounded. We also construct a family of aperiodic recurrent words such that their window complexity $P^w(n)$ is in $O(n^\alpha)$, while at least $n^\alpha$ for infinitely many $n$, where $0<\alpha<1$. Finally, we establish that the window complexity of a uniformly recurrent aperiodic word is unbounded.\\
%\end{center}



\section{Introduction}

Study of infinite words occupies a prominent place in the field of combinatorics on words. Combinatorics on words began one century ago with the seminal work of A. Thue \cite{19, 20} and it developed extensively in the last five decades \cite{3, 5, 15, 18}, see also references in  \cite{7}. Moreover, it interacts with many fields like Number Theory, Geometry, Dynamical Systems,  Theoretical Computer Science and Physics  (see \cite{3, 5, 14}). In order to get a better  understanding of an infinite word numerous concepts and combinatorial tools where introduced. In this paper,  we focus particularly on two of them: recurrence of factors  and complexity functions. An infinite word is said to be recurrent if all of its factors occur infinitely many times. In the literature, several notions of complexity function are met. They  share the same characteristic  consisting in  counting the number of  words (respecting possibly a certain property) appearing in a given infinite word. The first of them, called subword complexity, was introduced in 1975 by Ehrenfeucht and al. \cite{10}. It counts the number of subwords of a given length in an infinite word (see chapter 4 in \cite{7} for more results on the subject). The reader can consult \cite{2} and \cite{8} for a survey on  notions of complexity. In \cite{9},   a new notion of complexity is introduced: window complexity. The window complexity of an infinite word counts the number of different words obtained when factorising the infinite word into contiguous factors of a given length. In this paper we pursue investigation on window complexity and some recurrence aspects.

The paper is organized as follows. First, we fix some definitions   and notations,  (Section 2). In Section 3 and 4, the notions of uniformly and strongly modulo-recurrent word are respectively introduced. Some results about these notions are given. Section 5 is devoted to the relationship between the window complexity and the classical complexity in the case of the Thue-Morse word. In Section 6, we give a positive answer to one of the three questions pointed out in \cite{9}. Then,  we construct some aperiodic recurrent words such that their window complexity is in $O(n^\alpha)$ where $0<\alpha<1$ (Section 7). Finally, in Section 8, we show that the window complexity of a uniformly recurrent aperiodic word is unbounded.

\section{Preliminaries}

Let $\mathcal{A}^{*}$ be the free monoid generated by a non-empty finite set $\mathcal{A}$ called alphabet. The elements of $\mathcal{A}$ are called letters and those of $\mathcal{A}^{*}$, words. For any word $v$ in $\mathcal{A}^{*}$, $|v|$ denotes the length of $v$, namely the number of its letters. The identity element of $\mathcal{A}^{*}$ denoted by $\varepsilon$ is the empty word; it is the word of length $0$.

An infinite word is a sequence of letters in $\mathcal{A}$ indexed by $\mathbb{N}$. The set of infinite words on $\mathcal{A}$ is denoted $\mathcal{A}^{\omega}$ and the set of finite or infinite words, $\mathcal{A}^{\infty}=\mathcal{A}^{*}\cup \mathcal{A}^{\omega}$.

 We say that an infinite word $u$ is  periodic if there exists a finite word $v$   such that $u=vvv\cdots$; then $u$ is simply denoted $v^\omega $. In other terms, there exists a positive integer $\tau$ such that  $u_{i+\tau}=u_i$ for all $i\geq 0$.

A finite word $u$ of length $n$ formed by repeating a single letter $x$ is simply denoted~$x^n$. The $n$-th power of a finite word $v$, denoted $v^n$, is defined as the concatenation of $n$ copies of~$v$. An infinite word~$u$ is said to be eventually periodic if there exists two finite words $v$ and $w$ such that $u=wvvv\cdots$. In this case, we write $u=wv^\omega $. An infinite word which is not eventually periodic is called aperiodic.

Let $u\in \mathcal{A}^{\infty}$ and $v\in \mathcal{A}^{*}$. The word $v$ is called a factor of $u$ if there exist $u_1\in \mathcal{A}^{*}$ and $u_2\in \mathcal{A}^{\infty}$ such that $u=u_1vu_2$.

For any infinite word $u$ in $\mathcal{A}^{\omega}$, we shall write $u=u_0u_1u_2u_3\cdots$ where $u_i\in A$, for all $i\geq 0$.

 Let $u\in \mathcal{A}^{\omega}$. The language of length $n$ of $u$, denoted by $\mathcal{F}_n(u)$, is the set of factors of $u$ of length $n$:
$$\mathcal{F}_n(u)=\left\{u_ku_{k+1}\cdots u_{k+n-1}: k\geq0\right\}\enspace.$$
The set of all the factors of $u$ is denoted  $\mathcal{F}(u)$. A factor $v$ of length~$n$ of a word $u=u_0u_1u_2\cdots$ appears in $u$ at  position $k$ if $v=u_ku_{k+1}\cdots u_{k+n-1}$. %A word~$u$ is said to be recurrent if any factor of $u$ appears infinitely many times in $u$.

\begin{definition}
An infinite word $u$ is said to be
\begin{itemize}
\item recurrent if every factor of $u$ appears infinitely many times in $u$,%often,
\item  uniformly recurrent   if for all $n\in \mathbb{N}$, there exists $N\in \mathbb{N}$ such that every factor of $u$ of length $N$ contains all of factors of $u$ of length $n$.
\end{itemize}
\end{definition}
For more details on these types of words we refer the reader to   \cite{7, 15} for instance.

A morphism $f$ is a map from $\mathcal{A}^{*}$ to $\mathcal{B}^{*}$ (where $\mathcal{B}$ is an alphabet possibly different from $\mathcal{A}$) such that $f(uv)=f(u)f(v)$ for all $u,\,v\in \mathcal{A}^{*}$.

An infinite word $u$ is said to be  generated by a morphism $f$ if there exists a letter $x\in \mathcal{A}$ such that the words $x$, $f(x)$, $f^2(x)$, $\ldots$, $f^n(x)$, $\ldots$ are longer and longer prefixes of $u$. Then, we denote   $u=f^{\omega}(x)$.

The complexity function of an infinite word $u$ is the map from $\mathbb{N}$ to $\mathbb{N}^{*}$ defined by $P_{u}( n)=\#\mathcal{F}_n(u)$, where $\#\mathcal{F}_n(u)$ is the number of elements in $\mathcal{F}_n(u)$.

Let us recall the classical result due to Morse and Hedlund.
%\begin{theorem}
%Let $u$ be an infinite word. The following properties are equivalent
%\begin{enumerate}
%\item $u$ is eventually periodic,
%\item The complexity function of $u$ is bounded.
%\end{enumerate}
%\end{theorem}

\begin{theorem} (\cite{16}) \label{morse-hedlund} Let  $u$ be an infinte word. Then, the following statements are equivalent:
\begin{enumerate}

\item The infinite word $u$ is eventually periodic,

\item There  exists an integer  $n$ such that $p(n) =p(n+1)$,

\item There  exists an integer  $n$ such that  $p(n)\leq n$,

\item The sequence $(p(n))_n$ is bounded.

\end{enumerate}
\end{theorem}

A Sturmian word $u$ is an infinite word such that $P_{u}(n)=n+1$ for every integer $n\geq0$. Sturmian words are aperiodic infinite words of minimal complexity. These words have been widely studied; we refer the reader to \cite{15, 18} for more details.

A Sturmian morphism is a morphism $f$ defined on a binary alphabet such that for any Sturmian word $u$, $f(u)$ is also Sturmian. Sturmian morphisms are characterized in \cite{berst-seeb}.

We say that $v$ appears in $u$ at   position $\equiv i\mod k$ ($i$ modulo $k$) if there exists $l\in \mathbb{N}$ such that $v=u_{lk+i}  u_{lk+i+1}\cdots u_{lk+i+|v|-1}$~.

\begin{definition}
A factor $v$  of length  $n\geq 2$ of an infinite word $u$ is said to be a $n$-window factor of $u$ if $v$  appears in $u$ at a position  multiple of $n$, $i.e$ $\equiv 0\mod n$. \end{definition} 

 The set of $n$-window  factors of $u$, is denoted  $\mathcal{F}^{w}_n(u)$ : 
$$\mathcal{F}^{w}_n(u)=\left\{u_{kn}u_{kn+1}\cdots u_{(k+1)n-1}: k\geq 0\right\}\enspace.$$

\begin{definition} Let $u=u_0u_1u_2\cdots$ be an infinite word. The window complexity function of $u$ is the map $P^{w}_{u}: \mathbb{N}\longrightarrow \mathbb{N}^*$ defined by
$$P^{w}_{u}(n)=\# \mathcal{F}^{w}_n(u) \enspace.$$
\end{definition}

For a fixed infinite word $u$, the following comparision is obvious:
\begin{eqnarray}
 P^{w}_{u}\leq P_{u} \label{Pf<=P}
\end{eqnarray}

In the sequel, when there is no ambiguity we simply  denote $P(n)$, $P^w(n)$, $\mathcal{F}_{n}$, $\mathcal{F}^{w}(n)$ instead of  $P_u(n)$, $P^{w}_{u}(n)$, $\mathcal{F}_{u}(n)$,  $\mathcal{F}^{w}_{u}(n)$. 

Let $\mathcal{A}$ be an alphabet with $q$ letters. A word $u$ in $\mathcal{A}$ for which $\mathcal{F}(u)=\mathcal{A}^{*}$ is called a maximal complexity word. Its complexity is $P(n) =q^n$. Examples of such words are constructed as follows.  It is well-known that the monoid $\mathcal{A}^{*}$ is countable and can be ordered. Let $\mathcal{A}^{*}=\left\{w_0,\, w_1,\, w_2,\,\cdots\right\}$ be an enumeration on $\mathcal{A}^{*}$. Then, the infinite word $u=w_0w_1w_2w_3\ldots$, obtained by concatenating all words $w_i$,  is called Champernowne word and it has maximal  complexity, since by construction every word in $\mathcal{A}^{*}$ occurs in it.

\section{Modulo-recurrent words}
Modulo-recurrent words were introduced in \cite{12}. 
Let us make a brief survey on these words:

\begin{definition} An infinite word $u=u_0u_1u_2\cdots$ is said to be modulo-recurrent if, for any $k\geq1$, every factor $w$ of $u$ appears in $u$ at every position modulo $k$, \ie,
$$\forall w\in\mathcal{F}(u), \forall k\geq 1, \forall i\in\left\{0,\,1,\,\ldots,\,k-1\right\},\ \exists l\in\mathbb{N} : w=u_{kl+i}u_{kl+i+1}\cdots u_{kl+i+|w|-1}\enspace.
$$
\end{definition}

Modulo-recurrent words  are recurrent words. In \cite{12} and \cite{9} respectively it is shown that Sturmian words and maximal complexity words are modulo-recurrent.

In the second paper the authors obtained a characterization of modulo-recurrent words in terms of window complexity as follows: %wordnow study the window complexity of a particular class of infinite words, introduced in \cite{kt1}: modulo-recurrent words.

\begin{theorem}
\label{thm-modulo}
Let $u$ be a recurrent infinite word. Then, the following assertions are equivalent:
\begin{enumerate}
\item The word $u$ is modulo-recurrent.
\item $\forall n\geq 0,\ P_f(u,\, n)= P(u,\,n)$.
\end{enumerate}
\end{theorem}

To prove that a given word is modulo-recurrent, the following lemma allows to check only particular factors (such as prefixes).

\begin{lemma} \label{modulo-prefix}
Let $u$ be an infinite word and $L$ be a subset of $\mathcal{F}(u)$ such that 
$$\forall w\in \mathcal{F}(u), \exists v\in L : w\in \mathcal{F}(v).$$
Assume that, for any $k\geq 1$ and $v\in L$, $v$ appears in $u$ at every position modulo $k$. Then $u$ is modulo-recurrent.
\end{lemma}

\begin{proof}
Let $w$ be a factor of $u$, and choose $v\in L$ such that $w$ occurs at  some position $j$ in $v$. Since $v$ appears  in $u$ at every position modulo $k$, it is the same for $w$.

\end{proof}

\section{Uniformly modulo-recurrent words}

In this section, we introduce the notion of uniformly modulo-recurrent words. It is a uniform version of modulo-recurrence, in a  similar way as unform recurrence  is a uniform version of recurrence.
%\begin{definition} An infinite word $u$ is said to be
 % uniformly recurrent   if for all $n\in \N$, there exists $N$ such that any factor of $u$ of length $N$ contains any factor of $u$ of length $n$.

%\end{definition}

\begin{definition} An infinite word $u=u_0u_1u_2\cdots$ is said to be uniformly modulo-recurrent if, for any $k\geq1$ and $n\geq 0$ there exists $N$ such that every factor $w$ of $u$ of length $n$  appears in any factor $v$ of $u$ of length $N$ at every position modulo $k$.

\end{definition}

\begin{proposition} Let $u$ be an infinite word. Then, the following statements are equivalent

\begin{enumerate}
	\item  $u$ is  uniformly recurrent and modulo-recurrent.
	\item  $u$ is uniformly modulo-recurrent.
\end{enumerate}
\end{proposition}

\begin{proof}
Clearly $(2)$ implies $(1)$. It remains to show that $(1)$ implies $(2)$. Let $k\geq 1$ and $n\geq 0$. Let $w\in \mathcal{F}_n(u)$. Since $u$ is modulo-recurrent, for all $i\in\left\{0,\ldots, k-1\right\}$ there exists $m_{w,\,i}\in \mathbb{N}$ such that $w=u_{km_{w,\,i}+i} \cdots u_{km_{w,\,i}+i+n-1}.$
 So,  take
  $$m=\underbrace{\max}_{w\in \mathcal{F}_{n}(u),\, i\in\left\{0,\,\ldots, k-1\right\}}(km_{w,\,i}+i)\  \textrm{ and}\ z=u_0u_1\ldots u_{m+n-1}.$$
  Then, for any $i\in\left\{0,\, \ldots,\, k-1\right\}$ and for any $w\in\mathcal{F}_n(u)$, $z$ contains at least one occurrence of $w$ at a position $\equiv i\mod k$. Furthermore, since $u$ is uniformly recurrent, there exists $N\in\mathbb{N}$ such that any factor $v$ of length $N$ contains also the factor $z$:
$$\forall v\in\mathcal{F}_N(u): z\in \mathcal{F}(v).$$
Hence, for any $i\in\left\{0,\, \ldots,\, k-1\right\}$ and for any $w\in\mathcal{F}_n(u)$, $v$ contains one occurrence of $w$ at position $\equiv i\mod k$. So,  $u$ is uniformely modulo-recurrent. \end{proof}

\begin{remark}  
\begin{enumerate}
	\item Modulo-recurrence  does  not imply   uniform modulo-recurrence. For example, a Champernowne word is modulo-recurrent but it is not uniformly recurrent.
	\item Uniform  recurrence does not imply uniform modulo-recurrence. For example,  the Thue-Morse word is uniformly recurrent but it is not modulo-recurrent.
	\end{enumerate}
\end{remark}

\section{Strongly modulo-recurrent words}

Here, we identify a subclass of modulo-recurrent  words such that elements remain  modulo-recurrent under action of morphisms respecting  a particular condition. %it is to define some modulo-recurrent words recurrent words getting by action of some morphisms which respect some condition.

\begin{definition}
An infinite  word $u$ is said to be strongly modulo-recurrent  if for any morphism  $\sigma$ on $\mathcal{A}^{*}$ verifying the following property
$$\forall a, b\in \mathcal{A}: a\neq b \Rightarrow  gcd(\left|\sigma(a)\right|,\left|\sigma(b)\right|)=1$$
then,   the word $\sigma(u)$ is modulo recurrent.
\end{definition}

\begin{lemma} \label{Rauzy-pairs} (Rauzy's pair construction \cite{17}). Let $\mathcal{A}=\left\lbrace a,b \right\rbrace$ and $x$, $y$ be two coprime positive integers. Then there exists a Sturmian morphism $f$ such that $|f(a)|=x$ and $|f(b)|=y$.

\end{lemma}

\begin{proof}
By induction on $x+y$. If $x=y=1$, let $f=Id$. Otherwise, assume that for instance that $x>y$ and let $x'=x-y$ and $y'=y$ as in the additive version of Euclid's algorithm. By induction hypothesis, there is a Sturmian morphism $f'$ such that 
$|f'(a)|=x'$ and $|f'(b)|=y'$. Let $f(a)=f'(ab)$ and $f(b)=f'(b)$ then $f$ is Sturmian by \cite{17}.
\end{proof}

\begin{theorem} \label{smr}
Every Sturmian word is strongly modulo-recurrent.

\end{theorem}

\begin{proof}  
Let $u$ be a Sturmian word on $\mathcal{A}=\left\{a,\,b\right\}$ and  $\tau$ a morphism on $\mathcal{A}^{*}$ such that $\left|\tau(a)\right|$ and $\left|\tau(b)\right|$ are coprime.

 By Lemma \ref{Rauzy-pairs},   we can find a Sturmian morphism $\sigma$ such that for all $x\in \mathcal{A}$,  $\left|\sigma\left(x\right)\right|=\left|\tau\left(x\right)\right|$. So, $\sigma\left(u\right)$ is Sturmian. Thus,  $\sigma\left(u\right)$ is modulo-recurrent. Let $z$ be  the longest common prefix of $\sigma(ab)$ and $\sigma(ba)$, and $z'$ be their longest common suffix. It is known \cite{17} that $|zz'|=|\sigma(ab)|-1$. Moreover, if $w$ is a factor of $u$, then $z'\sigma(w)z$ is a factor of $\sigma(u)$, and $z'\sigma(w)z$ occurs at some position $n$ in $\sigma (u)$. So, if we put  $u=pws$  then we have $|\sigma(p)|=n+|z'|$.%\cite{desubstitution sturmienne}.
  Let $w$ be a factor of $u$. Then, for all  $k>1$ and $i\geq 0$, $\sigma\left(w\right)$ appears  in  $\sigma\left(u\right)$ at some position $n$ with $n\equiv i\mod k$ since $\sigma\left(u\right)$ is modulo-recurrent. So, we can write $u=pws$ with $p$ verifying $\left|\sigma\left(p\right)\right|=n$. Since, $\tau(u)=\tau(p)\tau(w)\tau(s)$ with $\left|\tau\left(p\right)\right|=n$, then $\tau\left(w\right)$ appears in $\tau\left(u\right)$ at the position $n$. \end{proof}

\begin{theorem}
Champernowne  words are strongly modulo-recurrent.
\end{theorem}
\begin{proof}  
Let $u$ be a Champernowne binary word and  $\tau$ a morphism on $\mathcal{A}^{*}$ such that $gcd\left(\left|\tau(a)\right|,\left|\tau(b)\right|\right)=1$. Let $w$ be a factor of $u$,  $m$ and $i$ some integers. Write $\left|\tau(a)\right|=x$, $\left|\tau(b)\right|=y$ and consider the set 
$$E=\left\{\left|\tau\left(w\right)\right|+nx+py: n,\, p\in \mathbb{N}\right\}.$$
 Since $gcd(x,\,y)=1$, there exist $u,\, v\in\mathbb{Z}$ such that $ux+vy=1$. So, the set $E$ is such that $E=\mathbb{N}- H$ where $H$ is a finite subset of $\mathbb{N}$. Thus, there exists $n_0$ such that for all $l\geq n_0$ one has $l\in E$. Let $j=\left|\tau\left(w\right)\right|+nx+py \in E$ such that $j\equiv 1 \mod m$. Since the word $w'=\left(wa^nb^p\right)^m$ occurs  in $u$ then $\tau(w')$  appears in $\tau(u)$ at some position $m_0$. It follows that $\tau(w)$ appears in $\tau(u)$ at the  positions $m_0$, $m_0+j$, $m_0+2j$, $\ldots$, $m_0+(m-1)j$. Among these positions, one is $\equiv i\mod m$ because $j\equiv 1\mod m$. \end{proof}

\begin{remark}  
If a word $u$ is strongly modulo-recurrent then $u$ is modulo-recurrent.
\end{remark}

Indeed, for any strongly modulo-recurrent word $u$, $Id_{\mathcal{A}}(u)$ is modulo-recurrent, where $Id$ is the identity morphism.
Note that the converse of this remark is false. The  following example shows it. Consider the morphisms $\sigma_1: a\mapsto aa,\ b\mapsto b$ and $\sigma_2: a\mapsto a,\ b\mapsto bb$. We have $\sigma_2\circ\sigma_1: a\mapsto aa,\ b\mapsto bb$. Let $u$ be a sturmian word. Then, by Theorem \ref{smr} $\sigma_1(u)$ is modulo-recurrent. However $\left(\sigma_2\circ\sigma_1\right)(u)=\sigma_2\left(\sigma_1\left(u\right)\right)$ is not modulo-recurrent. So, the word $\sigma_1(u)$  is modulo-recurrent but not strongly modulo-recurrent.

%It is well-known that in the subclass of modulo-recurrent words, window complexity coincides with classical complexity.  Since $P^{w}_{u}\leq P_u$, it seems interesting to look  finely at this inequality in the other cases, at least for some classical words.

 \section{About  window complexity of the Thue-Morse word}

%The Thue-Morse  word, $\mathbf{t}$, is one of the classical infinite words which has been the most studied. It is   generated by the morphism  $\theta : a\mapsto ab,\, b\mapsto ba$. 

The Thue-Morse  word, $\mathbf{t}$, is one of the classical infinite words which has been the most studied. It is   generated by the morphism  $\theta : a\mapsto ab,\, b\mapsto ba$. An equivalent definition of $\mathbf{t}$ is:
\begin{definition} The Thue-Morse word $\mathbf{t}=\mathbf{t}_0\mathbf{t}_1\mathbf{t}_2\cdots$ is defined recursively by $\mathbf{t}_0=a$, $\mathbf{t}_1=b$ and for all $n\geq 0$, $\mathbf{t}_{2n}=\mathbf{t}_n$, $\mathbf{t}_{2n+1}=\overline{\mathbf{t}_{n}}$, where $\overline{a}=b$ and vice-versa.

\end{definition}

From this definition it follows that for all $j,\,n\in \mathbb{N}$, $ \mathbf{t}_{2^{j}n}=\mathbf{t}_n$ and if $0\leq n<2^j$ then $\mathbf{t}_{2^{j}+n}=\overline{\mathbf{t}_{n}}.$

For more information on this famous word we refer the reader to \cite{1,2, 4, 6}. 

In Section 2, we have seen that the subclass of modulo-recurrent words is the set of words for which window complexity of the elements coincides with  classical complexity.  Since $P^{w}_{u}\leq P_u$, it seems interesting to look  finely at this inequality in the other cases, at least for some classical words. In this section we see what happens about the window complexity of the Thue-Morse word.

One has:

 \begin{theorem} \label{TM} The window complexity of $\mathbf{t}$ satisfies:
 $$\forall n\in \mathbb{N},\ P^w(n)=P\left(\frac{n}{2^{\upsilon_2\left(n\right)}}\right)$$
 where $\upsilon_2\left(n\right)$ denotes diadic valuation of $n$. In particular, we have $ P^w(n)=P(n)$ if $n$ is odd.

 \end{theorem}

\begin{proof} 
 $\bullet$ Observe that for all $n\in \mathbb{N}$, window factors of length $2n$ are images by $\theta$ of window factors of length $n$ of $\mathbf{t}$: $\mathcal{F}^{w}_{2n}(\mathbf{t})=\theta\left(\mathcal{F}^{w}_{n}(\mathbf{t})\right)$. So,  as $\theta$ is injective,
  $P^{w}( 2n)=P^{w}( n)$. Iterating the process we get
 $$\forall n\in \mathbb{N},\ P^w( n)=P^{w}\left(\frac{n}{2^{\upsilon_2\left(n\right)}}\right).$$
                       $\bullet$ Let $n$ be an   odd integer. We have to show that
  $\ P^w(n)=P\left(n\right).$
%   Since $n$ is odd then we can write $n=2q+1$ where $q\in \mathbb{N}$. Furthermore, there exists an integer $j$ such that $2^j=1\mod n$ with $2^j>2n$. In other terms, there exists $l\geq 2$ such that $2^j=ln+1$.

It suffices to get that any factor of $u$ of length $n$ is an $n$-window factor of $\mathbf{t}$. Let $v\in \mathcal{F}_n(\mathbf{t})$. Then, there exists $m\in \mathbb{N}$ such that $v=\mathbf{t}_{m}\mathbf{t}_{m+1}\ldots \mathbf{t}_{m+n-1}$. Let us show that $v\in \mathcal{F}^{w}_{n}(\mathbf{t})$. 
   Since $n$ is odd, then we can write $n=2q+1$ where $q\in \mathbb{N}$. Furthermore, there exists an integer $j$ such that $2^j\equiv 1\mod n$ with $2^j\geq m+2n$. In other terms, $v$ appears at position $m$ in $\theta^{j}(a)$.

Let   $l\in \mathbb{N}$ be such that $2^j=ln+1$. Then, we have:
 $$(l+1)n=(2^j-1)+n=2^j+2q \ \mathrm{and} \ (l+2)n=2^j+2n-1=2^j+4q+1.$$
  So, we have  $\mathbf{t}_{(l+1)n}=\overline{\mathbf{t}_{2q}}$ and $\mathbf{t}_{(l+2)n}=\overline{\mathbf{t}_{4q+1}}=\mathbf{t}_{2q}$ since  $2^j>2q$ and $2^j>4q+1$.
 %Similarly, we have
 %$$(l+2)n=2^j+2n-1=2^j+4q+1$$
 %Hence $t_{\left(l+2\right)n}=\overline{t_{4q+1}}$
 %since $2^j>4q+1$.
 Thus, we have $\mathbf{t}_{(l+2)n}=\overline{\mathbf{t}_{(l+1)n}}$. By applying $\theta^j$ to $\mathbf{t}$ we get $\mathbf{t}_{2^{j}(l+2)n+i}=\overline{\mathbf{t}_{2^{j}(l+1)n+i}}$ for all  $0\leq i< 2^j-1$. Choose $i$  such that $i+m=0\mod n$.  So, there exists $k\in \lbrace2^{j}(l+1),\, 2^{j}(l+2)\rbrace$ such that $\mathbf{t}_{kn+i}=a$. It follows that
 $\mathbf{t}_{2^{j}(kn+i)}\ldots \mathbf{t}_{2^{j}\left( kn+i\right)+2^j-1 }=\theta ^{j}(a)$.
 
  Therefore $v$ appears in $\mathbf{t}$ at position $2^{j}(kn+i)+m$. 
 Now, $2^{j}(kn+i)+m\equiv i+m \equiv 0\mod n$.
 That means $v$ is an $n$-window factor of $\mathbf{t}$: $v\in \mathcal{F}^{w}_{n}(\mathbf{t})$.
 
% Now, in $\mathbf{t}$ it is known that $t_{4q+1}=\overline{t_{2q}}$. Thus, $t_{\left(l+2\right)n}=\overline{t_{\left(l+1\right)n}}$.  It follows that
% $$\forall 0\leq i<2^j :\ t_{2^j\left(l+2\right)n+i}=\overline{t_{2^j\left(l+1\right)n+i}}.$$
% Hence $t_{kn+i}=a$ with $k\in\left\{2^j\left(l+1\right),\, 2^j\left(l+2\right)\right\}$.
%
%So, we deduce that any factor of $\mathbf{t}$ is in $\mathcal{F}^{w}_{n}(\mathbf{t})$.
  \end{proof}

From Theorem \ref{TM} we deduce that for the Thue-Morse word, the  set $\Pi=\left\{n\in \mathbb{N}: P^{w}(n)=P(n) \right\}$ is infinite.

 By this result  and   inequality (\ref{Pf<=P}) it becomes interesting to ask the following question:
 Does there exist some words for which the window complexity satisfies:
 $$\forall n\geq 2, \ P^{w}(n)<P(n)?$$
 Besides giving a response to a question asked in \cite{9}, we provide in the next Section a positive answer to the above question. 
 
\section{An example of recurrent aperiodic word with bounded window-complexity}

  With Tapsoba,  the authors asked  in \cite{9} the following question: Does there exist infinite recurrent and aperiodic words for which  the window complexity function is bounded?

Our aim in this section is to construct a recurrent word with bounded window complexity in order to give an answer to the above question.

Consider the sequence of positive integers $\left(k_n\right)$ defined by $k_0=1$, and  $k_{n+1}= \left(2k_n\right)!$ for all $n\geq 0$. 

The first few terms of $(k_n)$  are:
$$
\begin{array}{|r|c|c|c|c|c|c|}
	\hline
	n	&0&1&2&3&4&\cdots\\
	\hline
	k_n	&1&2&24&48!& \left(2\times 48!\right) ! &\cdots\\
	\hline
\end{array}
$$
 The sequence $\left(k_n\right)$ is strictly increasing.

\begin{lemma}
$\forall i\geq 1:\ k_i>\sum_{j=0}^{i-1}k_j$
\end{lemma}

\begin{proof} 
For $i=1$ we have $k_1=2>k_0=1$. If $ k_i>\sum_{j=0}^{i-1}k_j$, then
$$k_{i+1}=\left(2k_i\right)!\geq 2k_i>k_i+\sum_{j=0}^{i-1}k_j.$$
We conclude by induction on $i$. \end{proof}

\begin{lemma} \label{divide}
Let $i$ and $ n$ be two positive integers such that $2k_{i-1}<n\leq 2k_i$. Then $n$ divides $k_{i+1}$.
\end{lemma}

\begin{proof} Since $k_{i+1}= \left(2k_i\right)!$ and $n\leq 2k_i$ then $n$ divides $k_{i+1}$.  \end{proof}

 Now, define a sequence of words $x_i$ related to the sequence  $(k_i)$ as follows: $x_0=1$ and $x_{i+1}=x_ix_i0^{k_{i+1}-2k_{i}}$.  
$$	\begin{array}{lcl}
		x_0&=&1\\
			x_1&=&11\\
			x_2&=&11110^{20}\\
			x_3&=&11110^{20}11110^{48!-28}\\
				x_4&=&11110^{20}11110^{48!-28}11110^{20}11110^{\left[ 2\left(48!\right) \right] !- 48!-28}\\
			\vdots
		\end{array}$$
		
	The sequence	$(x_i)$ tends to an infinite word $u=u_0u_1u_2 \cdots$:
	 $$u=\lim_{i\rightarrow \infty} x_i=11110^{20}11110^{48!-28}\ldots$$ 

 Observe that  $u=u_0u_1u_2\cdots$ can be defined equivalently  by $u_0=1$ and $u_n=1$  when there exists $I\subset \mathbb{N}$ such that $n=\sum_{i\in I} k_i $ and $u_n=0$ otherwise. 

\begin{remark} \label{wind-F_k_i}
We have $|x_i|=k_i$ and  
 $$\forall i\geq 0: u\in \lbrace x_i,\, 0^{k_{i}} \rbrace^{\omega} .$$ 
 In other terms, $\mathcal{F}^{w}_{k_{i}}=\lbrace x_i, 0^{k_{i}}\rbrace$.
 \end{remark}
\begin{remark}
\begin{enumerate}
	\item The word $u$ is recurrent
	\item The word $u$ is not eventually periodic.
\end{enumerate}
 \end{remark}
 
Indeed, it is sufficient to observe that $x_i$ is a prefix of $u$ and appears at least twice  in $u$ for all $i\geq 0$.
The second statement is obtained by observing  that $u$  contains  factors of the form $0^{k_{i}-2k_{i-1}}1$ where $(k_{i}-2k_{i-1})$ is an unbounded sequence.

The table below  gives the first few values of %$\mathcal{F}^{w}_{n} (u)$ and
  $P^{w}_{u}(n).$\\  
 $$
\begin{array}{|r|c|c|c|c|c|c|c|c|c|c|c|c|c|c|c|c}
	\hline
	n	&1&2&3&4&5 &6&7&8&9&10&11&12&13&14&15&\cdots\\
	\hline
	P^{w}_{u}(n)&2&2&3&2&4 &2&3&2&4&3&3&2&4&3&3&\cdots\\
	\hline
\end{array}
$$
$$\begin{array}{r|c|c|c|c|c|c|c|}
	\hline
	\cdots &23&24&25&26&27&28&\cdots\\
	\hline
	\cdots&3&2&3&3&3&2&\cdots\\
	\hline
\end{array}
$$
%\begin{array}{|r|c|c|c|c|c|c|}
%	\hline
%	n	&1&2&3&4&5 &6&7&8&9&10&11&12&13&14&15&16\\
%	\hline
%	P^{w}_{u}(n)	&1&2&3&4&5 &6&7&8&9&10&11&12&13&14&15&16\\
%	\hline
%\end{array}
%$
%
%$$
%\begin{array}{|c|l|l|}
%	\hline
%	n	&\mathcal{F}^w_{n}(u)&P^{w}_{u}(n)\\
%	\hline
%	1	&\left\{0,\, 1\right\}&2\\
%	\hline
%	2	&\left\{00,\, 11\right\}&2\\
%	\hline
%	3	&\left\{000,\, 100,\, 111\right\}&3\\
%	\hline
%	4	&\left\{0000,\, 1111\right\}&2\\
%	\hline
%	5	&\left\{11110,\,00000,\, 00001,\,  11100\right\}&4\\
%	\hline
%	6	& \left\{111100,\,000000\right\}& 2\\
%	\hline
%	7	&\left\{1111000,\,0000000,\, 0001111\right\}&4\\
%	\vdots	&\cdots&\cdots\\
%	\hline
%\end{array}
%$$

\begin{proposition}
The window complexity of $u$ satisfies
$$\forall n\in \mathbb{N}, \ P^w(n)\leq 4.$$
\end{proposition}

\begin{proof} Let $n$ be an integer, $n\geq 3$. There exists $i$ such that $2k_{i-1}< n\leq 2k_i$. Then, $n$ divides $k_{i+1}$ by Lemma \ref{divide}.

Since $n$ divides $k_{i+1}$, it results from Remark \ref{wind-F_k_i} above that 
$$\mathcal{F}^{w}_{n}(u)=F^{w}_{n}(x_{i+1})\cup F^{w}_{n}(0^{k_{i+1}})$$
But
\begin{eqnarray}
F^{w}_{n}(x_{i+1})&=&F^{w}_{n}(x_{i}x_{i}0^{k_{i+1}-2k_{i}}) \\ 
                  &=&F^{w}_{n}(x_{i-1}x_{i-1}0^{k_{i}-2k_{i-1}}x_{i-1}x_{i-1}0^{k_{i}-2k_{i-1}}0^{k_{i+1}-2k_{i}})
\end{eqnarray}

Let $k_i\equiv \beta \mod n$, with $0\leq \beta< n$, and $\alpha=n-2k_{i-1}$.

If $n\leq k_{i}$, we get the following.

- If $\beta \in \left\{0,\, 1,\, \ldots,\,n-2k_{i-1}\right\}$ we have  
$$ F^{w}_{n}(u=\left\{x_{i-1}x_{i-1}0^{\alpha},\, 0^{\beta} x_{i-1}x_{i-1}0^{\alpha-\beta},\,  0^n\right\}.$$

- If $\beta \in \left\{n-2k_{i-1}+1,\, \ldots,\,n-1\right\}$ we have  
$$ F^{w}_{n}(u)=\left\{x_{i-1}x_{i-1}0^{\alpha},\,  0^{\beta}Pref_{n-\beta}(x_{i-1}x_{i-1}),\, Suf_{\beta-\alpha}(x_{i-1}x_{i-1})0^{n+\alpha-\beta},\, 0^n\right\}$$

If $n> k_{i}$, then $\beta=k_i$ and we get:

- If $n<k_i+2k_{i-1}$  
$$ F^{w}_{n}(u)=\left\{x_{i-1}x_{i-1}0^{k_{i}-2k_{i-1}}Pref_{n-k_{i}}(x_{i-1}x_{i-1}),\, Suf_{k_i-\alpha}(x_{i-1}x_{i-1})0^{n+\alpha-\beta},\,  0^n\right\}.$$

- If $n\geq k_i+2k_{i-1}$  
$$ F^{w}_{n}(u)=\left\{x_{i-1}x_{i-1}0^{k_{i}-2k_{i-1}}x_{i-1}x_{i-1}0^{n-k_{i}-2k_{i-1}},\,  0^n\right\}.$$

So, it follows the inequality
$$
\ P^{w}(n)\leq 4.$$  
\end{proof}

\begin{corollary} The window complexity of $u$ satisfies:
$$\forall n\geq 2, \ P^w(n)<P(n).$$
\end{corollary}

\begin{proof} Since $u$ is non-eventually periodic, by Theorem \ref{morse-hedlund} we have 
$$\forall n\geq 0, \ P(n)\geq n+1. $$
 But, we have $ P^w(2)=2<3\leq P(2), \  P^w(3)=3<4\leq P(3)$ and for all $n\geq 4$, using the previous Proposition $ P^w(n)\leq 4<5\leq P(n)$. Thus, the corollary is valid.

 \end{proof}

\section{Window complexity of a  familly of recurrent words}

To contrast the previous Section   we provide in this Section a familly of recurrent words with unbounded window complexity.
\begin{theorem}

Let $0<\alpha<1$.  There exists a recurrent infinite word $u$ such that:
\begin{enumerate}
\item $P^{w}_{u}(n)=O(n^\alpha)$,
\item For infinitely many $n$, $P^{w}_{u}(n)\geq n^\alpha$.
\end{enumerate}
\end{theorem}

\begin{proof} Define the sequences of integers $(l_i)$ and $(k_i)$ as follows: $l_0=2$, $k_0=1$ and $l_{i+1}=\Big\lceil\left(2l_ik_i\right)^{\alpha}\Big\rceil$, $k_{i+1}=2\left(l_ik_i\right)!$. Note that $(l_i)$ is non-decreasing. \\
Consider now the sequence of words $(x_i)$ defined  as  follows: 
$$x_0=1,\ x_{i+1}=x_{i}^{l_i}0^{k_{i+1}-l_{i}k_{i}}$$
Observe that $|x_i|=k_i$ and $x_i$ is a proper prefix of $x_{i+1}$. So,  $(x_i)$ converges to an infinite word $u$. We have also  $u\in\left\{x_i,\,0^{k_{i}}\right\}^\omega$, for all $i$.

\begin{enumerate}

\item Let $n\geq 3$. Choose $i\geq 1$ such that  $l_{i-1}k_{i-1}<n\leq l_ik_i$. Since $n$ divides $k_{i+1}=2(l_ik_i)!$, we have:
$$F^{w}_{n}(u)=F^{w}_{n}(x_{i+1})\cup F^{w}_{n}(0^{k_{i+1}}).$$
But $x_{i+1}=\left(x_{i-1}^{l_{i-1}}0^{k_{i}-l_{i-1}k_{i-1}}\right)^{l_i}0^{k_{i+1}-l_{i}k_{i}}$ and  each occurence of $x_{i-1}^{l_{i-1}}$ intersects at most two windows %of $\left(x_{i-1}^{l_{i-1}}0^{k_{i}-l_{i-1}k_{i-1}}\right)^{l_i}0^{k_{i+1}-l_{i}k_{i}}$
 of length $n$ (the first one intersects only one window). The windows that do not intersect $x_{i-1}^{l_{i-1}}$ are $0^n$.\\
Therefore, $  P_{u}^{w}(n)\leq  1+2(l_i-1)+1=2l_i=2\Big\lceil \left(2l_{i-1}k_{i-1}\right)^{\alpha}\Big\rceil\leq 2\Big\lceil \left(2n\right)^{\alpha}\Big\rceil$.
So, $P_w(n)=O(n^\alpha)$.

\item Let $i\geq 1$. By Bertrand-Chebyshev theorem (Theorem 418, p. 343 in  \cite{11}), there exists a prime number $n$ such that
$$ l_{i-1}k_{i-1}<n\leq2l_{i-1}k_{i-1}.$$ 
Then $\gcd(n,\, k_i)=1$, and $$F^{w}_{n}(u)\supseteq \left\{0^{jk_i\mod n} Pref_{n-(jk_i\mod n)}\left(x_i\right): 0\leq j\leq \min(l_i,\,n)-1\right\}$$
These words begin with $0^{jk_i\mod n}1$. So, they are all distinct. Therefore, $P^{w}_{u}(n)\geq \min(l_i,\,n)=\min\left(2\Big\lceil \left(2l_{i-1}k_{i-1}\right)^{\alpha}\Big\rceil,\, n\right)\geq \min\left(\Big\lceil n^{\alpha}\Big\rceil,\, n\right)\geq n^{\alpha}.$

\end{enumerate}. \end{proof}

\section{Window complexity and uniformly recurrent words}

Here, a property of the window complexity of a  uniformly recurrent word related to   periodicity is given.% using their window complexity.
\begin{theorem} \label{liminf-fini=periodic}

Let $u$ be a uniformly recurrent word such that
 $$\liminf_{n\rightarrow \infty} P^{w}_{u}(n) P^{w}_{u}(n+1) <\infty/.$$
 Then,  $u$ is periodic.

\end{theorem}

\begin{proof}
Let $u$ be a uniformly recurrent word. If $\liminf P^{w}_{u}(n) P^{w}_{u}(n+1) <\infty$ then,  there exist an integer $M$ and an  increasing sequence $(n_k)$ such that $ P^{w}_{u}(n_{k}) P^{w}_{u}(n_{k}+1) \leq M$. %It follows $ P^{w}_{u}(n_{k})\leq M$ and $ P^{w}_{u}(n_{k}+1)\leq M$.

% In this stape, as in the uniformely recurrent case we show   that $ P^{w}_{u}(n_{k}) $ admits periods with arbitrary size; also to $ P^{w}_{u}(n_{k}+1)$. Thus $u$ is periodic.

%Let $u$ be a uniformly recurrent aperiodic word. Suppose that $P^w(n)$ is bounded, $ie$
%$$\exists M>0: \forall n\in \mathbb{N}, P^w(n)\leq M.$$
Let  $N$ be a term of   $(n_k)$ large enough. One has $P^w(N)\leq M$ and $P^w(N+1)\leq M$. Recall that  $P^w(n)=\#\mathcal{F}^{w}_{n}(u)$. For all $i\in \left\{0,1, \, \ldots, \, N-1\right\}$, let us define the words $x_i$, $y_i$, and $z_i$ as follows:
$$\begin{array}{lcl}

x_i&=&u_{Ni}\cdots u_{Ni+N-1}\\
y_i&=&u_{N\left(i+1\right)}\cdots u_{N\left(i+1\right)+N-1}\\
z_i&=&u_{\left(N+1\right)i}\cdots u_{\left(N+1\right)i+N}\\

\end{array}$$

Then, we have $ x_i,\, y_i\in \mathcal{F}^{w}_{N}\left(u\right)$, $ z_i\in \mathcal{F}^{w}_{N+1}\left(u\right) $ and $z_i$ is a factor of $x_iy_i$. Furthermore, observe that $z_i$ appears in $x_iy_i$ at position $i$.

Thus, it results the following inclusion
$$E=\left\lbrace\left(x_i,y_i, z_i\right):  0\leq i\leq N-1\right\rbrace \subseteq\mathcal{F}_{N}^{w}(u)\times \mathcal{F}_{N}^{w}(u)\times \mathcal{F}_{N+1}^{w}(u).$$

So, $\# E \leq M^3$. Therefore, there exists some triple $(x, y, z)\in E$  for which we have $\# I\geq \frac{N}{M^{3}}$ where $I=\left\lbrace i\in \left\lbrace 0,\, \cdots,\, N-1\right\rbrace : \left( x_i,y_i,z_i\right) =\left( x,y,z\right)  \right\rbrace $. Let us denote by  $|xy|_z$ the number of occurrences of $z$ in $xy$. Since,  $i\in I$  means that $z$ appears in $xy$ at position $i$, we have $|xy|_z\geq \frac{N}{M^3}$ and also
$$\min\left\{j-i: i,\, j\in I, \, i<j\right\}\leq \frac{N-1}{\#I-1}\leq \frac{N-1}{\frac{N}{M^3}-1}$$ 

Now, assume that $N>M^6$. Then $\frac{N}{M^3}-M^3>0$ and  $(\frac{N}{M^3}-1)(M^3+1)>N-1$. Thus, for $N>M^6$ we have $\min\left\{j-i: i,\, j\in I, \, i<j\right\}<M^3+1$.
 %and since $\left(\frac{N}{M^3}-1\right)\left(M^3+1\right)>N-1$ if $N> M^6$.

In other terms,  there exist $i,\,j\in I$ such that $0<j-i\leq M^3$~. Consequently, $z$ is $(j-i)$-periodic. Thus, we have

% So, $|xy|_z\geq \frac{N}{M^3}$ and
%
%$$\min\left\{j-i: i,\, j\in I, \, i<j\right\}\leq \frac{N-1}{\#I-1}\leq \frac{N-1}{\frac{N}{M^3}-1}<M^3+1$$
%since $\left(\frac{N}{M^3}-1\right)\left(M^3+1\right)>N-1$ if $N> M^6$.
%
%Then, there exist $i,\,j\in I$ such that $0<j-i\leq M^3$~. Consequently, $z$ is $(j-i)$-periodic. Thus, we have
$$\forall N>M^6,\exists q\in \mathcal{F}(u): 1\leq |q|\leq M^3\ \textrm{and}\  q^{\left\lfloor \frac{N+1}{\left|q\right|}\right\rfloor}\in \mathcal{F}(u).$$
As there are finitely many such $q$, then there exists $q\in \mathcal{F}(u),\, 1\leq |q|\leq M^3$ and an  infinite subset $K$ of $\mathbb{N}$ such that $q^k$ occurs in $u$ for all $k\in K$.
 
So
$$ \exists q\in L(u),\,   \forall k\in \mathbb{N},\ q^k \in L(u).$$

Thus, either $L(u)=L(q^\omega) $ and then $u$ is $|q|$-periodic; or there exists  $r\in L(u)-L(q^\omega)$ and $r\notin L(q^k)$ for all $k\in \mathbb{N}$. 
This latter  case contradicts uniform recurrence of $u$.  \end{proof}

Theorem \ref{liminf-fini=periodic} leads to the following corollary.

\begin{corollary}
Let $u$ be a uniformly recurrent aperiodic word. Then, the window complexity of $u$ is unbounded.
\end{corollary}

\begin{proof}

Let $u$ be a uniformly recurrent aperiodic word. Suppose that $P^w(n)$ is bounded, $i.e.$,
$$\exists M>0: \forall n\in \mathbb{N}, P^w(n)\leq M.$$

Then $\liminf_{n\rightarrow \infty} P^{w}_{u}(n) P^{w}_{u}(n+1) <\infty$. So, by Theorem \ref{liminf-fini=periodic},  $u$ is periodic. This is absurd because $u$ is aperiodic

%Let be $N$ a positive integer  large enough. One has $P^w(N)\leq M$ and $P_w(N+1)\leq M$. From here %(At this stape)
%  as in the uniformely recurrent case we show   that $ P^{w}_{u}(n_{k}) $ admits periods with arbitrary size; also to $ P^{w}_{u}(n_{k}+1)$. Thus $u$ is periodic.
% 
\end{proof}

\bibliographystyle{plain}

\begin{thebibliography}{99}
%\addcontentsline{toc}{chapter}{\protect\numberline{}{References}}

%\bibitem{j-p-al} J.-P.~Allouche, Sur la complexit\'e des suites infinies, \textit{Bull. Belg. Math. Soc.} \textbf{1} (1994), 133--143.
 \bibitem{1} J.-P.  Allouche, Thue,  Combinatorics  on  words,  and  conjectures  inspired  by  the  Thue-Morsesequence, J. Th\'eor. Nombres Bordeaux \textbf{27}(2015), 375-388.


\bibitem{2} J.-P.~Allouche, Surveying some notions of complexity for finite and infinite sequences, in Functions in Number Theory and Their Probabilistic Aspects, Eds. RIMS, Kyoto University {\bf B34} (2012), 27-38.

\bibitem{3} J.-P.~Allouche, J. Shallit. \emph{Automatic Sequences: Theory, Applications, Generalizations}, Cambridge University Press, UK, 2003.

\bibitem{4} J.-P. Allouche, J. Shallit, The ubiquitous Prouhet-Thue-Morse sequence, Sequences and their applications (Singapore, 1998), Discrete Math. Theor.
Comput. Sci., Springer, London, 1999, pp. 1-16

\bibitem{5} J. Berstel, Sturmian and episturmian words (a survey of some recent results), In CAI 2007, LNCS {\bf 4728} , 23-47, 2007.

%\bibitem{berst-seeb1} J. BERSTEL, P. S\'E\'EBOLD, \emph{A characterization of Sturmian morphisms}, in: A. Borziskowski. S. Sokolowski (Eds.). MFCS'93. Lecture Notes in Computer Science 711 (1993) 281-290.

\bibitem{berst-seeb} J. Berstel, P. S\'e\'ebold, \emph{Morphismes de Sturm}, Bull. Belg. Math. Soc. Simon Stevin 1 (1994) 175-189. Journ\'ees Montoise (Mons, 1992).

\bibitem{6} S. Brlek, Enumeration of the factors in the Thue-Morse word, Discr. Appl. Math. 24 (1-3) (1989), 83-96.

%\bibitem{al-baak-cass-damnik} J.-P.~Allouche, M. Baake, J. Cassaigne, D. Damanik, Palindrome complexity, \textit{Theoret. Comput. Sci.} \textbf{292} (2003), 9--31.

%\bibitem{av-frid} S. V. Avgustinovich, D. G. Fon-Der-Flaass, A. E. Frid, \textit{Arithmetical complexity of infinite words}, in:\ Words, Languages and Combinatorics III, World Scientific, Singapore, 2001.

%\bibitem{cassaigne} J. Cassaigne, Complexit\'e et facteurs sp\'eciaux, \textit{Bull. Belg. Math. Soc.} \textbf{4} (1997), 67--88.

%\bibitem{cobham} A. Cobham, Uniform tag sequences, \textit{Math. Systems Theory} \textbf{6} (1972), 164--192.

\bibitem{7} Cant, \emph{Combinatorics, Automata and Number Theory}, V. Berth\'e, M. Rigo (Eds),
Encyclopedia of Mathematics and its Applications 135,  Cambridge University Press (2010). 
 
 \bibitem{8}   J. Cassaigne, I. Kabor\'e, \textit{Abelian complexity and frequencies of letters in infinite words}, Int. J. Found. Comput. Sci., Vol. \textbf{27(}05), pp. 631- 649, 2016.

\bibitem{9} J. Cassaigne, I. Kabor\'e, T. Tapsoba. \emph{On a new notion of complexity on infinite words}, Acta Univ. Sapientiae, Mathematica \textbf{2}(2): 127-136, 2010.


%\bibitem{caka} J. Cassaigne, I. Kabor\'e, \emph{Etude de la complexité du mot de Fibonacci g\'en\'eralis\'e}. in: Proceedings of $11^{th}$ African  Conference on Research in Computer Science and Applied Mathematics, CARI'12,    pp 62-69, 2012.


\bibitem{10} A. Ehrenfeucht, K. P. Lee, G. Rozenberg, Subword complexities of various classes of deterministic developmental languages without interaction, \textit{Theoret. Comput. Sci.} \textbf{1} (1975) 59--75.

\bibitem{11} G. H. Hardy, E. M. Wright, An Introduction to theory of Numbers, Fourth Ed. Oxford University Press, 1960.

\bibitem{12}I. Kabor\'e, T. Tapsoba, Combinatoire de mots r\'ecurrents de complexit\'e $n+2$, \textit{RAIRO-Theoret. Inform. Appl.} \textbf{41} (2007), 425--446.

\bibitem{13} I. Kabor\'e, T. Tapsoba, \textit{Complexit\'e par fen\^etre}, in:\ Proceedings of the $6^{\rm th}$ International Conference on Words (WORDS 2007), pp. 184--188, Institut de Math\'ematiques de Luminy, Marseille (France), 2007.



\bibitem{14} J. Leroy, G. Richomme, \emph{A combinatorial proof of S-adicity for sequences with linear complexity}, INTEGERS, vol 13,  \#A5, 2013. 

\bibitem{15} M. Lothaire, \textit{Algebraic combinatorics on words}, Cambridge University Press, 2002.

\bibitem{16} M. Morse, G. A. Hedlund, \emph{Symbolic dynamics}, Amer. J. Math. \textbf{60} (1938), 815-866.
%\bibitem{loth} M. Lothaire, \emph{Algebraic combinatorics on words}, Cambridge University Press, 2002.


\bibitem{17} F. Mignosi, P. S\'e\'ebold, \emph{Morphismes sturmiens et r\`egles de Rauzy}; J. Th\'eor.  Nombres  Bordeaux, \textbf{5},(1993), 221-233.

%\bibitem{pansiot} J.-J. Pansiot, Complexit\'e des facteurs des mots infinis engendr\'es par morphismes it\'er\'es, in:\ ICALP'84, pp. 380--389, \textit{Lect. Notes Comput. Sci.} \textbf{172}, Springer-Verlag, 1984.

\bibitem{18} N. Pytheas Fogg, \textit{Substitutions in Dynamics, Arithmetics and Combinatorics}, Lect. Notes Math. 1794, Springer-Verlag, 2002.

%\bibitem{taps} T. Tapsoba, Suites infinies, complexit\'e et g\'eom\'etrie, \textit{Rev. CAMES, S\'erie A},  \textbf{6} (2008), 94--96.

\bibitem{rauzy} G. Rauzy, \emph{Suites \`a termes dans un alphabet fini}. S\'emin.  Th\'eorie des nombres (1982-1983) 25-01. 25-16. Bordeaux.

\bibitem{19} A. Thue, \"Uber die gegenseitige Lage gleicher Teile gewisser Zeichenreihen, \textit{Norske Vid. Selsk. Skr. I. Math. Nat. Kl. Christiana} \textbf{1} (1912), 1--67.


\bibitem{20} A. Thue, \"Uber unendliche Zeichenreihen, \textit{Norske Vid. Skr. I. Kl. Christiana} \textbf{7} (1906), 1--22.


\end{thebibliography}

\end{document}